\documentclass[reqno]{amsart}

\usepackage{hott-arxiv}

\title{The join construction}
\date{\today}
\author{Egbert Rijke}
\address{Department of Philosophy\\
Carnegie Mellon University\\
Pittsburgh, PA  15213 }
\email{erijke@andrew.cmu.edu}

\begin{document}

\maketitle

\begin{abstract}
In homotopy type theory we can define the join of maps as a binary operation on 
maps with a common co\-domain.
This operation is commutative, associative, and the unique map from the empty type into the
common codomain is a neutral element. 
Moreover, we show that the idempotents of the join of maps are precisely the 
embeddings,
and we prove the `join connectivity theorem', which states that the 
connectivity of the join of maps equals the join of the connectivities of
the individual maps.

We define the image of a map $f:A\to X$ in $\UU$ via the join construction,
as the colimit of the finite join powers of $f$. 
The join powers therefore provide approximations of the image inclusion,
and the join connectivity theorem implies that the 
approximating maps into the image increase in connectivity.

A modified version of the join construction can be used to show that for
any map $f:A\to X$ in which $X$ is only assumed to be locally small, the
image is a small type. 
We use the modified join construction to give an alternative construction of
set-quotients, the Rezk completion of a precategory, and we define
the $n$-truncation for any $n:\N$.
Thus we see that each of these are definable operations on a univalent universe
for Martin-L\"of type theory with a natural numbers object, that is moreover
closed under homotopy coequalizers.

\end{abstract}

\tableofcontents

\section{Introduction}
Homotopy type theory extends Martin-L\"of's dependent type theory
\cite{MartinLof84}, with ideas from Awodey and Warren's homotopical interpretation
of identity types \cite{AwodeyWarren} and Voevodsky's construction of the
model of Martin-L\"of type theory with the univalence axiom in the simplicial sets \cite{KapulkinLeFanuLumsdaine}.
The univalence axiom was proposed by Voevodsky in \cite{Voevodsky06,Voevodsky10}.
In addition to the univalence axiom, the addition of higher inductive types
was proposed by Lefanu-Lumsdaine \cite{Lumsdaine11Blog} and Shulman \cite{Shulman11Blog}.
The consequences of the univalence axiom and the properties of higher inductive
types were further explored by the Univalent Foundations Program in \cite{hottbook}.
We refer to \cite{hottbook} for a further background to the subject of homotopy
type theory.

The present article is concerned with a new construction of the image of a map
$f:A\to X$ that we call the `join construction', in a univalent universe for Martin-L\"of type theory that is
assumed to be closed under homotopy pushouts. In particular, we do not assume
an operation of propositional truncation as described in \S 6.9 of \cite{hottbook}.
The question of constructing the propositional truncation in this setting
has also be addressed by Van Doorn \cite{VanDoorn} and Kraus \cite{Kraus}.
We show how the join construction can be modified to accommodate for the
case where $X$ is only assumed to be locally small, and we use this modified
version of the join construction to construct set-quotients, the Rezk-completion
of a pre-$1$-category, and $n$-truncations for every $n:\N$. 
Finally, we show that the sequence approximating the image of $f:A\to X$ increases in connectivity.

For the remainder of this introduction, we discuss the prerequisites of our
work and the methods that we deploy. 
Let us first state precisely the setting in which we work.
We assume a univalent universe $\UU$ in Martin-L\"of type 
theory, containing the usual types:
the empty type 
$\emptyt$, the unit type $\unit$, a natural numbers object $\N$, cartesian 
products $A\times B$, dependent function types $\prd{x:A}P(x)$, disjoint sums 
$A+B$, dependent pair types $\sm{x:A}P(x)$, and identity types $x=_A y$.
Recall that the univalence axiom implies function extensionality for $\Pi$-types
in the universe. 
We will assume function extensionality globally, i.e.~for $\Pi$-types of any size.
Moreover, we will assume that $\UU$ closed under \emph{graph quotients}. 
A model for this setting is the cubical set model of Bezem, Coquand and Huber
\cite{BezemCoquandHuber}, although it should be remarked that the graph
quotients have not yet been accounted for.

We recall graph quotients briefly here, but for more details we refer to
the forthcoming article \cite{RijkeSpitters}. 

\begin{defn}
A \define{(small) graph} $\Gamma$ is a pair $(\pts{\Gamma},\edg{\Gamma})$ consisting of
a type $\pts{\Gamma}:\UU$ of \define{vertices}, 
and a type-valued binary relation $\edg{\Gamma}:\pts{\Gamma}\to \pts{\Gamma}\to\UU$
of \define{edges}.
\end{defn}

\begin{defn}
For any small graph $\Gamma$, the \define{graph quotient} $\tfcolim(\Gamma)$ of
$\Gamma$ is a higher inductive type in $\UU$, with constructors
\begin{align*}
\pts{\mathsf{constr}} & : \pts{\Gamma} \to \tfcolim(\Gamma) \\
\edg{\mathsf{constr}} & : \prd{i,j:\pts{\Gamma}} \edg{\Gamma}(i,j)\to(\pts{\mathsf{constr}}(i)=\pts{\mathsf{constr}}(j))
\end{align*}
and satisfying the corresponding induction principle, which can be used to define sections
of type families of any size.
\end{defn}

\begin{rmk}
In the present work we shall only need to use the induction principle of graph
quotients to define sections of type families of locally small types, but we
postpone the definition of local smallness to \autoref{sec:modified-join}, 
where it becomes relevant.
\end{rmk}

Note that the graph quotient of a graph $\Gamma$ is just the homotopy 
coequalizer of the two projections $(\sm{i,j:\pts{\Gamma}}\edg{\Gamma}(i,j))\rightrightarrows\pts{\Gamma}$, 
and it is straightforward to obtain pushouts from these.
It therefore follows that whatever can be described using pushouts 
(e.g.~suspensions, the join, the smash product, and so on) can
also be obtained from these higher inductive types. 

Moreover, since we assume a natural numbers object $\N$ to be in $\UU$, 
we can also define sequential colimits from graph quotients. 
To see this, note that a type sequence
\begin{equation*}
\begin{tikzcd}
A_0 \arrow[r,"f_0"] & A_1 \arrow[r,"f_1"] & A_2 \arrow[r,"f_2"] & \cdots
\end{tikzcd}
\end{equation*}
is a pair $(A,f)$ consisting of
\begin{align*}
A & : \N\to\UU \\
f & : \prd{n:\N} A_n\to A_{n+1}.
\end{align*}
From $(A,f)$ we obtain the graph $\Gamma\jdeq (\Gamma_0,\Gamma_1)$ consisting of
\begin{align*}
\Gamma_0 & \defeq\sm{n:\N}A_n \\
\Gamma_1(\pairr{n,x},\pairr{m,y}) & \defeq \sm{p:n+1=m} \trans{p}{f_n(x)}=y. 
\end{align*}
From this description it follows that the only outgoing edge from
$\pairr{n,x}$ is the canonical edge from $\pairr{n,x}$ to $\pairr{n+1,f_n(x)}$.
The sequential colimit of $(A,f)$ is defined to be the graph quotient of this
graph $\Gamma$.

Because there is not yet a definitive, fully general formulation of higher 
inductive types, 
we restrict our attention to the special class of higher inductive
types that can be gotten from graph quotients. 
These are non-recursive, and therefore the induction
principle with the associated computation rule describes \emph{an equivalent way} of mapping
out of them into an arbitrary type. This is not the case for the current formulation
of some recursive higher inductive types. For example, the induction principle
of the propositional truncation only describes how to eliminate into another
mere proposition. In this sense, the graph quotients are a little better understood.

In \autoref{sec:join-maps}, we define the join of two maps with a common codomain
as the pushout of their pullback, as indicated in the following diagram
\begin{equation*}
\begin{tikzcd}
A\times_X B \arrow[r] \arrow[d] \arrow[dr, phantom, "{\ulcorner}", at end] & B \arrow[d] \arrow[ddr,bend left=15,"g"] \\
A \arrow[r] \arrow[drr,bend right=15,swap,"f"] & \join[X]{A}{B} \arrow[dr,densely dotted,swap,near start,"\join{f}{g}" xshift=1ex] \\
& & X.
\end{tikzcd}
\end{equation*}
Note that in the special case where $X\jdeq \unit$, the type $\join[\unit]{A}{B}$
is equivalent to the usual join operation $\join{A}{B}$
on types that is described in \cite{hottbook}.
In \autoref{defn:join-fiber} we will show that the join $\join{f}{g}$ of maps may be seen as the fiberwise
join. Analogously to the fiber product, we will write $\join[X]{A}{B}$ for the domain of $\join{f}{g}$ to signify
this specification.

To prove that the fibers of the join $\join{f}{g}$ are the join of the fibers
of $f$ and $g$, we will use the descent theorem from \cite{RijkeSpitters}. 
For pushouts, the descent theorem works as follows. 
Consider a span $\mathcal{S}_{f,g}$ given by
\begin{equation*}
\begin{tikzcd}
X & A \arrow[l,swap,"f"] \arrow[r,"g"] & Y,
\end{tikzcd}
\end{equation*}
and consider a map $h:Z\to (X+_A Y)$ into the pushout. By pulling back along $h$,
we obtain a cartesian map of spans, meaning that the evident squares are pullbacks, as indicated in the diagram
\begin{equation*}
\begin{tikzcd}
X \times_{(X+_A Y)} Z \arrow[d] & A\times_{(X+_A Y)} Z \arrow[l] \arrow[r] \arrow[d] \arrow[dr, phantom, "{\lrcorner}", at start] \arrow[dl, phantom, "{\llcorner}", at start] & Y\times_{(X+_A Y)} Z \arrow[d] \\
X & A \arrow[l,"f"] \arrow[r,swap,"g"] & Y.
\end{tikzcd}
\end{equation*}
We see that this describes a map 
$\mathsf{sliceToCart}:\UU/(X+_A Y)\to \mathsf{Cart}_{\mathcal{S}_{f,g}}$ 
from the type $\UU/(X+_A Y)$ of all maps into the pushout 
$X+_A Y$, to the type $\mathsf{Cart}_{\mathcal{S}_{f,g}}$ of all cartesian maps of spans into $\mathcal{S}_{f,g}$. 
The descent theorem, as we will use it, asserts that this map is an equivalence.
In fact, the descent theorem is itself equivalent to the univalence axiom 
\cite{RijkeSpitters}. 

The inverse of $\mathsf{sliceToCart}$
maps a cartesian map of spans to the map between their pushouts. Note that
any map of spans determines a map between their pushouts. The assumption of 
cartesianness ensures that if we start with a map of spans and apply
$\mathsf{sliceToCart}$ to the map between their pushouts, we get the original
span back. 

Starting with a map $h:Z\to X+_A Y$, we obtain a cube
\begin{equation*}
\begin{tikzcd}[column sep =3em]
& \makebox[3em]{$A\times_{(X+_A Y)} Z$} \arrow[rr,start anchor={[xshift=1.3em]},end anchor={[xshift=-1.3em]}] \arrow[dl] \arrow[dd] 
  & & \makebox[3em]{$Y\times_{(X+_A Y)} Z$} \arrow[dd] \arrow[dl] \\
\makebox[3em]{$Y\times_{(X+_A Y)} Z$} \arrow[rr,crossing over,start anchor={[xshift=1.3em]}] \arrow[dd] & & Z  \\
&  A \arrow[dl,swap,"f"] \arrow[rr,"g" near start] & & Y \arrow[dl,end anchor={[xshift=.5ex,yshift=.5ex]}] \\
X \arrow[rr,end anchor={[xshift=-1.3em]}] & & \makebox[.8em]{$X +_A Y$} \arrow[from=uu,crossing over,"h" near start]
\end{tikzcd}
\end{equation*}
in which the bottom square is the original pushout square, and the vertical squares are all pullback squares.
By the descent theorem it follows that the top square is again a pushout square. In this sense, the descent theorem provides a way of commuting pushouts with pullbacks. This observation is sometimes also called the `flattening lemma' for pushouts, and its formalization in homotopy type theory is due to Brunerie \cite{BruneriePhD}.

With the join of maps available, we consider in \autoref{sec:join-construction} 
join-powers $f^{\ast n}$, which are iterated joins of $f:A\to X$ with itself. 
The sequential colimit $f^{\ast\infty}$ of the join powers $f^{\ast n}$ turns
out to be equivalent to the image inclusion of $f$, as we will show in \autoref{thm:image}.

In the presence of
propositional truncation, every map in homotopy
type theory can be factored through a surjective map followed by an embedding,
in a unique way (see for instance Chapter 7 of \cite{hottbook}). 
The usual definition of surjectivity of a map $f:A\to X$ involves
propositional truncation
\begin{equation*}
\mathsf{isSurj}(f)\defeq\prd{x:X}\brck{\hfib{f}{x}},
\end{equation*}
and also the image itself is defined using propositional truncation,
as the type
\begin{equation*}
\sm{x:X}\brck{\hfib{f}{x}}.
\end{equation*}
However, in this article we do \emph{not} assume propositional truncation. 
Instead, we characterize the image of a map via its universal property with
respect to embeddings and in \autoref{thm:image} we will \emph{define} an embedding with the universal 
property of the image of $f$, as the infinite join power $f^{\ast\infty}$. 

In the special case where $X\jdeq \unit$,
this defines the propositional truncation of $A$. More precisely, we show that
the infinite join power $A^{\ast\infty}$ is the propositional truncation of $A$.

Let us state here the universal property of the image of $f$ with respect to embeddings.
Recall that an embedding is a map for which the homotopy fibers are mere propositions.
For any two maps $f:A\to X$ and $g:B\to X$ with a common codomain, we may consider
the type
\begin{equation*}
\mathrm{Hom}_X(f,g)\defeq \sm{h:A\to B} f\htpy g\circ h
\end{equation*}
of maps $h:A\to B$ such that the triangle
\begin{equation*}
\begin{tikzcd}
A \arrow[rr,"h"] \arrow[dr,swap,"f"] & & B \arrow[dl,"g"] \\
& X,
\end{tikzcd}
\end{equation*}
commutes. First, we observe that in the case where $g:B\to X$ is an embedding, it follows
that $\mathrm{Hom}_X(f,g)$ is a mere proposition. 
To see this, we apply the type theoretic principle of
choice, which is sometimes refered to as $\choice{\infty}$, to compute
\begin{align*}
\sm{h:A\to B} f\htpy g\circ h
& \eqvsym \prd{a:A}\sm{b:B} f(a)=g(b) \\
& \jdeq \prd{a:A}\fib{g}{f(a)}.
\end{align*}
This is a product of mere propositions, and mere propositions are closed under
dependent products. Thus, we see that any given $f:A\to X$ factors through
an embedding $g:B\to X$ in at most one way.

If we are given a second map $f':A'\to X$ with a commuting triangle
\begin{equation*}
\begin{tikzcd}
A \arrow[rr,"i"] \arrow[dr,swap,"f"] & & A' \arrow[dl,"{f'}"] \\
& X,
\end{tikzcd}
\end{equation*}
we can precompose factorizations of $f'$ through $g$ with $i$ to obtain
a factorization of $f$ through $g$. Explicitly, we have a map
\begin{equation}\label{eq:precomp_triangle}
\varphi^g_{i,I}:\mathrm{Hom}_X(f',g)\to\mathrm{Hom}_X(f,g)
\end{equation}
given by
\begin{equation*}
\varphi^g_{i,I}(h,H)\defeq \pairr{h\circ i, \lam{a}\ct{I(a)}{H(i(a))}},
\end{equation*}
where $I:f\htpy f'\circ i$ is the homotopy witnessing that $f$ factors through $f'$.

\begin{defn}\label{defn:universal}
Let $f:A\to X$ be a map. The \define{image} of $f$ is a quadruple $(\im(f),i_f,q_f,Q_f)$
consisting of a type $\im(f)$, an embedding $i_f:\im(f)\to X$, and a commuting triangle
\begin{equation*}
\begin{tikzcd}
A \arrow[rr,"q_f"] \arrow[dr,swap,"f"] & & \im(f) \arrow[dl,"{i_f}"] \\
& X
\end{tikzcd}
\end{equation*}
where $Q_f:f\htpy i_f\circ q_f$ witnesses that the triangle commutes, satisfying 
\define{the universal property of the image} that for
every embedding $g:B\to X$, the canonical map
\begin{equation*}
\varphi^g_{q_f,Q_f} : \mathrm{Hom}_X(i_f,g)\to\mathrm{Hom}_X(f,g)
\end{equation*}
defined in \autoref{eq:precomp_triangle} is an equivalence.
\end{defn}

Note that, since $\varphi^g_{e,E}$ is a map between mere propositions, to prove
that it is an equivalence it suffices to find a map in the converse direction.
As we shall see in the join construction, it is sometimes useful to consider
the universal property of the image of $f$ without requiring that $m$ is an
embedding. For example, in \autoref{lem:factor_join}
we will show that this universal property is stable under the operation $\join{f}{\blank}$
of joining by $f$. 

In the special case where $X\jdeq \unit$ an embedding $m:Y\to \unit$ satisfies
the universal property of the image of $f:A\to \unit$ precisely when for any
mere proposition $B$, the precomposition map $\blank\circ i : (Y\to B)\to(A\to B)$
is an equivalence. In this sense, the universal property of the image of a map
is a generalization of the universal property of the propositional truncation.
Indeed, the image of a map $f:A\to X$ can be seen as the propositional truncation
in the slice over $X$.

As a first application of the join construction, 
we show in \autoref{thm:idempotents} that the mere
propositions are precisely the `canonical' idempotents of the join operation. 
A canonical idempotent for the join operation is a type $A$ for which the
canonical map $\inl:A\to\join{A}{A}$ is an equivalence. Having such a canonical
equivalence allows one to show that $A= A^{\ast\infty}$, and since the type
$A^{\ast\infty}$ is shown to be the propositional truncation of $A$, it follows
that $A$ is a mere proposition. As a corollary, the embeddings are precisely
the canonical idempotents of the join operation on maps with a common codomain.
Since embeddings in homotopy type theory are subtypes,
the join of embeddings is just the union of subtypes. Thus, we see that
in particular the join generalizes the union $A\cup B$ of subtypes. 

The fact that the embeddings are precisely the canonical idempotents of
the join operation is somewhat reminiscent to the assertion in Theorem 4.1 of
\cite{AwodeyBauer}, where it is shown that the mere propositions are precisely
the cartesian idempotents.

In \autoref{sec:modified-join} we observe that when $A$ is small, and when $X$
is `locally small' in the sense that its identity types are small, then the
join construction can be modified slightly so that we are still able to
define an embedding (of which the domain is in $\UU$) with the universal
property of the image inclusion of $f$. 
Thus, we show in \autoref{thm:modified-join} that the image of a map $f:A\to X$
from a small type $A$ into a locally small type $X$ can be constructed
under the assumption that $\UU$ is a univalent universe with a natural 
numbers object, and closed under graph quotients. Moreover, this image is
again a small type. This is the main result of the present article.
 
Basic examples of types that satisfy the condition of local smallness include
any type in $\UU$, the univalent universe itself, mere propositions of any size,
and the exponent $A\to X$ for any $A:\UU$ and any locally small type $X$.
In particular, the image of any dependent type $P:A\to\UU$ is a type in $\UU$.
This image is sometimes called the \define{univalent completion} of $P$.

Using this modified version of the join construction we can construct
set-quotients following Voevodsky's large construction of set-quotients as
the image of an equivalence relation, and we can give an alternative construction
of the Rezk completion of a precategory as the image of the Yoneda embedding.

It is also worth observing that the constructions of the propositional truncation,
set-quotients and the Rezk completion bear great similarity: they \emph{all} take the
image of the Yoneda embedding. For instance, if $R:A\to A\to\prop$ is a 
$\prop$-valued equivalence relation on a type $A$, then $R$ may be considered
as the Yoneda embedding from the pre-$0$-category $A$ (with morphisms given by
$R$) into the $\prop$-valued presheaves.  
In other words, propositional truncation, set-quotients and the Rezk completion restricted to
pregroupoids fit in a hierarchy of increasing homotopical complexity, analogous
to the hierarchy of h-levels.

\medskip
\begin{tabular}{cll}
\emph{level} & \emph{equivalence structure} & \emph{quotient operation} \\
\hline \\
$-1$ & trivial relation & propositional truncation \\
$0$ & $\prop$-valued equivalence relation & set-quotient \\
$1$ & pre-$1$-groupoid structure & Rezk completion \\
$\vdots$ & \qquad$\vdots$ & \qquad$\vdots$ \\
$\infty$ & `pre-$\infty$-groupoid structure' & $\infty$-quotient
\end{tabular}

\medskip
\noindent%
Note that the above table suggests that a trivial relation on a type $A$, which
is a relation $R:A\to A \to\UU$ such that $R(a,b)$ is contractible for each
$a,b:A$, can be regarded as a pre-$(-1)$-groupoid structure on $A$, and that
a $\prop$-valued equivalence relation on $A$ can be regarded as 
a pre-$0$-groupoid structure on $A$.

We do not yet have a precise, satisfactory type theoretic formulation of the notion
of `pre-$\infty$-groupoid structure' on a type. Nevertheless, there is a non-trivial
class of examples that should fit in the $\infty$-th level proper, namely the
relation $x,y\mapsto \modal(x=y)$, for any modality $\modal$ on $\UU$, of
which the quotient operation is the modality $\modal^+$ of $\modal$-separated objects,
where a type is said to be $\modal$-separated if its identity types are $\modal$-modal.
The construction of the modality of $\modal$-separated objects will be given
in \cite{RijkeShulmanSpitters}. However, we will do a specific instance of that
construction here, namely the construction of the $(n+1)$-truncation from
the $n$-truncation. In \autoref{thm:truncation} we will use this to show that
for any $n\geq -2$, the $n$-truncation is a definable operation on a univalent
universe that is closed under graph quotients.

In the final section we show that the join of maps increases connectivity, as one would expect.
If a map $f:A\to X$ is $M$-connected and $g$ is $N$-connected for two types
$M$ and $N$, then the join $\join{f}{g}$ will be $(\join{M}{N})$-connected. A
special case of this result is that if $f$ is $m$-connected and $g$ is $n$-connected
for $m,n:\N$, then $\join{f}{g}$ is $(m+n+2)$-connected. A result in similar
spirit is the join extension theorem, \autoref{thm:join-extension}. This theorem asserts that if $f:X\to Y$
is an $M$-connected map, and $P:Y\to\UU$ is an $(\join{M}{N})$-local family of
types (see \autoref{defn:local}), then precomposition
by $f$ is an $N$-local map of type $(\prd{y:Y}P(y))\to(\prd{x:X}P(f(x))$.
We will use the join extension theorem to prove the universal property of $(n+1)$-truncation.
The join extension and connectivity theorems could be viewed as a first set of results about the interaction
of join with modalities, see \cite{RijkeShulmanSpitters}.

It should be noted, although it is not the subject of this article, that the
join construction also gives rise to the Milnor-construction of the principal
bundle over a topological group \cite{Milnor}. In the setting of homotopy type theory, we take
as $\infty$-groups the pointed connected types $\mathrm{pt}:\unit\to\mathbf{B}G$.
Then the Milnor-construction considers the iterated join-powers of $\mathrm{pt}$
with itself. In the special case where $\mathbf{B}G\jdeq K(\Z/2\Z,1)$ we obtain
the real projective spaces from the Milnor construction, and in the case where
$\mathbf{B}G\jdeq K(\Z,2)$ we obtain the complex projective spaces. Research
in this direction is joint work with Buchholtz, see \cite{classifying-types,realprojective}.

\section{The join of maps}\label{sec:join-maps}

\begin{defn}
Let $f:A\to X$ and $g:B\to X$ be maps into $X$. We define the type $\join[X]{A}{B}$ and the \define{join}\footnote{\emph{Warning}: By $\join{f}{g}$ we do \emph{not} mean the functorial action of the
join, applied to $(f,g)$.} $\join{f}{g}:\join[X]{A}{B}\to X$ of
$f$ and $g$, as indicated in the following diagram:
\begin{equation*}
\begin{tikzcd}
A\times_X B \arrow[r,"\pi_2"] \arrow[d,swap,"\pi_1"] \arrow[dr, phantom, "{\ulcorner}", at end] & B \arrow[d,"\inr"] \arrow[ddr,bend left=15,"g"] \\
A \arrow[r,swap,"\inl"] \arrow[drr,bend right=15,swap,"f"] & \join[X]{A}{B} \arrow[dr,densely dotted,swap,near start,"\join{f}{g}" xshift=1ex] \\
& & X.
\end{tikzcd}
\end{equation*}
\end{defn}

\begin{thm}\label{defn:join-fiber}
Let $f:A\to X$ and $g:B\to X$ be maps into $X$, and let $x:X$. Then there is
an equivalence
\begin{equation*}
\eqv{\hfib{\join{f}{g}}{x}}{\join{\hfib{f}{x}}{\hfib{g}{x}}}.
\end{equation*}
\end{thm}

\begin{proof}[Construction]
Recall that the fiber of the map $\join{f}{g}$ at $x:X$ can be obtained as
the pullback
\begin{equation*}
\begin{tikzcd}
\fib{\join{f}{g}}{x} \arrow[r] \arrow[d] & \unit \arrow[d,"x"] \\
\join[X]{A}{B} \arrow[r,swap,"\join{f}{g}"] & X.
\end{tikzcd}
\end{equation*}
By pulling back along the map $\fib{\join{f}{g}}{x}\to \join[X]{A}{B}$ we
obtain the following cube
\begin{small}
\begin{equation*}
\begin{tikzcd}[column sep=3em]
& \makebox[5em]{$\sm{a:A}{b:B}(f(a)=g(b))\times (g(b)=x)$} \arrow[rr,start anchor={[xshift=6.3em]}] \arrow[dl,densely dotted] \arrow[dd] 
  & & \hfib{g}{x} \arrow[dd] \arrow[dl] \\
\hfib{f}{x} \arrow[rr,crossing over] \arrow[dd] & & \hfib{\join{f}{g}}{x} \\
  & \sm{a:A}{b:B}f(a)=g(b) \arrow[dl] \arrow[rr] & & B \arrow[dl] \\
A \arrow[rr] & & \join[X]{A}{B} \arrow[from=uu,crossing over]
\end{tikzcd}
\end{equation*}
\end{small}%
in which he bottom square is the defining pushout of $\join[X]{A}{B}$. 
The front, right and back squares are easily seen to be pullback squares, by the pasting lemma of pullbacks. Hence the dotted map, being the unique map such that the top and left squares commute, makes the left square a pullback. Hence the top square is a pushout by the descent theorem or by the flattening lemma for pushouts.

However, to conclude the join formula we need to show that the square
\begin{equation*}
\begin{tikzcd}
\big(\hfib{f}{x}\big)\times\big(\hfib{g}{x}\big) \arrow[r,"\pi_2"] \arrow[d,swap,"\pi_1"] & \hfib{g}{x} \arrow[d] \\
\hfib{f}{x} \arrow[r] & \hfib{\join{f}{g}}{x}
\end{tikzcd}
\end{equation*}
is a pushout. This is be shown by giving a fiberwise equivalence of type
\begin{equation*}
\prd{a:A}{b:B}{p:g(b)=x} \eqv{(f(a)=x)}{(f(a)=g(b))}.
\end{equation*}
We then take this fiberwise equivalence to be post-composition with $\opp{p}$. 
\end{proof}

\begin{rmk}
The join operation on maps with a common codomain is associative up to homotopy (this was formalized by Brunerie, see Proposition 1.8.6 of \cite{BruneriePhD}), and it is a commutative operation on the generalized elements of a type $X$. Furthermore, the unique map of type $\emptyt\to X$ is a unit for the join operation.
\end{rmk}

 In the following lemma we will show that the join of embeddings is again an embedding. This is a generalization of the statement that if $P$ and
$Q$ are mere propositions, then $\join{P}{Q}$ is a mere proposition, and actually the more general statement reduces to this special case. Therefore, the embeddings form a `submonoid' of the `monoid' of generalized elements. The join $\join{P}{Q}$ on embeddings $P$ and $Q$ is the same as the union $P\cup Q$. In \autoref{thm:idempotents} below, we show that the mere propositions are precisely the idempotents for the join operation.

\begin{lem}
Suppose $f$ and $g$ are embeddings. Then $\join{f}{g}$ is also an embedding.
\end{lem}

\begin{proof}
By \autoref{defn:join-fiber}, it suffices to show that if $P$ and $Q$ are mere
propositions, then $\join{P}{Q}$ is also a mere proposition. It is equivalent
to show that $\join{P}{Q}\to\iscontr(\join{P}{Q})$. Recall that $\iscontr(\blank)$
is a mere proposition. So it suffices to show that
\begin{align*}
& P \to \iscontr(\join{P}{Q}) \\
& Q \to \iscontr(\join{P}{Q}).
\end{align*}
By symmetry, it suffices to show only $P\to\iscontr(\join{P}{Q})$. Let $p:P$. 
Then $P$ is contractible, and therefore the projection $P\times Q\to Q$ is an
equivalence. Hence it follows that $\inl:P\to \join{P}{Q}$ is an equivalence,
which shows that $\join{P}{Q}$ is contractible.
\end{proof}

\section{The join construction}\label{sec:join-construction}

The join construction gives, for any $f:A\to X$, an approximation of the image
$\im(f)$ by a type sequence. 
Before we give the join construction, we will show that the universal
property of the image of $f$, is closed under
the operation $\join{f}{\blank}$ of joining by $f$, and in \autoref{lem:factor_seq}
we will also show that this property is closed under sequential colimits.

Let $f:A\to X$ and $f':A'\to X$ be maps, and consider a commuting triangle
\begin{equation*}
\begin{tikzcd}
A \arrow[rr,"i"] \arrow[dr,swap,"f"] & & A' \arrow[dl,"{f'}"] \\
& X
\end{tikzcd}
\end{equation*}
with $I:f\htpy f'\circ i$. Then we also obtain a commuting triangle
\begin{equation*}
\begin{tikzcd}
A \arrow[rr,"\inl"] \arrow[dr,swap,"f"] & & \join[X]{A}{A'} \arrow[dl,"{\join{f}{f'}}"] \\
& X
\end{tikzcd}
\end{equation*}
We will write $C_l$ for the homotopy $f\htpy \join{f}{f'}\circ\inl$ witnessing that the
above triangle commutes.

\begin{lem}\label{lem:factor_join}
For every embedding $g:B\to X$, if the map
\begin{equation*}
\varphi^g_{i,I}: \mathrm{Hom}_X(f',g)\to\mathrm{Hom}_X(f,g)
\end{equation*}
defined in \autoref{eq:precomp_triangle} is an equivalence, then so is 
\begin{equation*}
\varphi^g_{\inl,C_l}: \mathrm{Hom}_X(\join{f}{f'},g)\to\mathrm{Hom}_X(f,g).
\end{equation*}
\end{lem}

\begin{proof}
Suppose that $g:B\to X$ is an embedding, and that $\varphi^g_{i,I}$
Since $\varphi^g_{\inl,C_l}$ is a map between mere propositions, it
suffices to define a map in the converse direction.
Some essential ingredients of the proof are illustrated in \autoref{fig}.
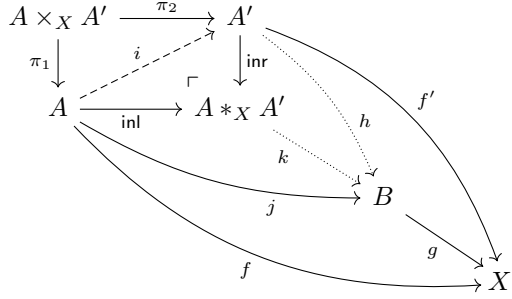
\begin{figure}
\begin{tikzcd}
A\times_X A' \arrow[r,"\pi_2"] \arrow[d,swap,"\pi_1"] \arrow[dr, phantom, "{\ulcorner}", at end] & A' \arrow[d,"\inr"] \arrow[ddr,bend left=15,densely dotted,near end,"h"] \arrow[dddrr,end anchor={[xshift=.2em]},bend left=25,"{f'}"] \\
A \arrow[ur,dashed,"i"] \arrow[r,swap,"\inl"] \arrow[drr,bend right=15,swap,near end,"j"] \arrow[ddrrr,bend right=25,end anchor={[xshift=.5em]},swap,"f"] & \join[X]{A}{A'} \arrow[dr,densely dotted,swap,near start,"k"] \\
& & B \arrow[dr,end anchor={[xshift=.5em,yshift=-.2em]},swap,"g"] \\
& & & \makebox[1.5em]{\centering $X$}
\end{tikzcd}
\caption{Diagram for the proof of \autoref{lem:factor_join}\label{fig}}
\end{figure}

Let $j:A\to B$ and $J:f\htpy g\circ j$. By our assumption on $f'$, we find
$h:A'\to B$ and $H:f'\htpy g\circ h$. Now it follows that the square
\begin{equation*}
\begin{tikzcd}
A\times_X A' \arrow[r] \arrow[d] & A' \arrow[d,"h"] \\
A \arrow[r,swap,"j"] & B
\end{tikzcd}
\end{equation*}
commutes, because that is equivalent (by the assumption that $g$ is an embedding)
to the commutativity of the square
\begin{equation*}
\begin{tikzcd}
A\times_X A' \arrow[r] \arrow[d] & A' \arrow[d,"{f'}"] \\
A \arrow[r,swap,"f"] & X.
\end{tikzcd}
\end{equation*}
Thus, we get from the universal property of $\join[X]{A}{A'}$ a map
$k:\join[X]{A}{A'}$ and homotopies $j\htpy k\circ\inl$ and $h\htpy k\circ\inr$.
It follows that $f\htpy (g\circ k)\circ \inl$ and $f'\htpy (g\circ k)\circ\inr$.
Hence by uniqueness we obtain a homotopy $K:\join{f}{f'}\htpy k\circ g$.
\end{proof}

\begin{lem}\label{lem:factor_seq}
Let $f:A\to X$ be a map, and consider a sequence $(A_n)_{n:\N}$ together with
a cone with vertex $A$ and a cocone with vertex $X$, as indicated in the
diagram 
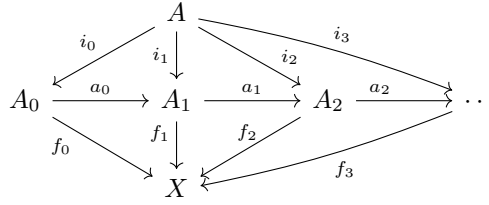
\begin{equation*}
\begin{tikzcd}[column sep=large]
& A \arrow[dl,swap,"i_0"] \arrow[d,swap,"i_1"] \arrow[dr,"i_2" near end] \arrow[drr,bend left=5,"i_3"] \\
A_0 \arrow[dr,swap,near start,"f_0"] \arrow[r,"a_{0}"] & A_1 \arrow[d,swap,near start,"f_1"] \arrow[r,"a_1"] & A_2 \arrow[dl,swap,"f_2" xshift=.5em] \arrow[r,"a_2" near start] & \cdots \arrow[dll,bend left=5,"f_3"] \\
& X
\end{tikzcd}
\end{equation*}
with colimit
\begin{equation*}
\begin{tikzcd}
A \arrow[r,"i_\infty"] & A_\infty \arrow[r,"f_\infty"] & X.
\end{tikzcd}
\end{equation*}
Let $g:B\to X$ be an embedding. If $\varphi^g_{i_n,I_n}$ is an equivalence for each $n:\N$, then so is $\varphi^g_{i_\infty,I_\infty}$. 
\end{lem}

\begin{proof}
To prove that $\varphi^g_{i_\infty,I_\infty}$ is an equivalence, it suffices to
find a map in the converse direction. Let $j:A\to B$ be a map, and let
$J:f\htpy g\circ j$ be a homotopy. Since each $\varphi^g_{i_\infty,I_\infty}$
is an equivalence, we find for each $n:\N$ a map
$h_n:A_n\to B$ and a homotopy $H_n: f_n\htpy g\circ h_n$. Then it follows that
the maps $(h_n)_{n:\N}$ form a cocone on $(A_n)_{n:\N}$ with vertex $B$, so
we obtain a map $h_\infty:A_\infty\to B$. It also follows that the homotopies
$(H_n)_{n:\N}$ form a compatible family of homotopies, so that we obtain
$H_\infty:f_\infty\htpy g\circ h_\infty$.
\end{proof}

\begin{thm}\label{thm:image}
In Martin-L\"of type theory with a univalent universe $\UU$ that is closed under
graph quotients we can define for every $f:A\to X$ with $A,X:\UU$ the image
of $f$ with $\im(f):\UU$.
\end{thm}

\begin{proof}[Construction]
Let $f:A\to X$ be a map in $\UU$.
First, we define a sequence
\begin{equation*}
\begin{tikzcd}
\im_\ast^0(f) \arrow[dr,swap,near start,"f^{\ast 0}"] \arrow[r,"i_{0}"] & \im_\ast^1(f) \arrow[d,swap,near start,"f^{\ast 1}"] \arrow[r,"i_1"] & \im_\ast^2(f) \arrow[dl,swap,"f^{\ast 2}" xshift=.5em] \arrow[r,"i_2"] & \cdots \arrow[dll,"f^{\ast 3}"] \\
& X.
\end{tikzcd}
\end{equation*}
We take $\im_\ast^0(f)\defeq \emptyt$, with the unique map into $X$. Then we take $\im_\ast^{n+1}(f)\defeq \join[X]{A}{\im_\ast^n(f)}$, and
$f^{\ast(n+1)}\defeq \join{f}{f^{\ast n}}$. The type $\im_\ast^n(f)$ is called the \define{$n$-th image approximation}, 
and the function $f^{\ast n}$ is called the \define{$n$-th join-power of $f$}. 

The \define{image} $\im(f)$ of $f$ is defined to be the sequential colimit
$\im_\ast^\infty(f)$. 
The embedding $i_f:\im(f)\to X$ is defined to be the map $f^{\ast\infty}$. 
Furthermore we have a canonical map $q_f:A\to \im(f)$ for which the triangle
\begin{equation*}
\begin{tikzcd}
A \arrow[dr,swap,"f"] \arrow[rr,"q_f"] & & \im(f) \arrow[dl,"i_f"] \\
& X
\end{tikzcd}
\end{equation*}
commutes. This satisfies the universal property of the image with respect
to embeddings by \autoref{lem:factor_join,lem:factor_seq}. Thus it remains to show that
$f^{\ast\infty}$ is an embedding, i.e.~that for any $x:X$, the type $\hfib{f^{\ast\infty}}{x}$ is a
mere proposition. 

Using the equivalence $\eqv{\isprop(T)}{(T\to\iscontr(T))}$ 
we can reduce the goal of showing that $\hfib{f^{\ast\infty}}{x}$ is a mere proposition, to
\begin{equation*}
\hfib{f^{\ast\infty}}{x}\to \iscontr (\hfib{f^{\ast\infty}}{x}).
\end{equation*}
To describe such a fiberwise map, it is equivalent to define a commuting triangle
\begin{equation*}
\begin{tikzcd}
\sm{x:X}\hfib{f^{\ast\infty}}{x} \arrow[rr] \arrow[dr,swap,"\proj 1"] & & \sm{x:X} \iscontr (\hfib{f^{\ast\infty}}{x}) \arrow[dl,"\proj 1"] \\
& X
\end{tikzcd}
\end{equation*}
Since $\iscontr(\blank)$ is a mere proposition, the projection on the right is an embedding. Since $f^{\ast\infty}$ satisfies the universal property of the image of $f$, we see that it is equivalent to show that
\begin{equation*}
\hfib{f}{x}\to \iscontr (\hfib{f^{\ast\infty}}{x}).
\end{equation*}
Let $a:A$ such that $f(a)=x$. then we need to show that $\hfib{f^{\ast\infty}}{f(a)}$ is contractible.
By Brunerie's flattening lemma, see \S 6.12 of \cite{hottbook}, it suffices to show that
\begin{equation*}
\tfcolim_n(\hfib{f^{\ast n}}{f(a)})
\end{equation*}
is contractible. By \autoref{defn:join-fiber}, it follows that $\hfib{f^{\ast n}}{f(a)}$ is equivalent
to $(\hfib{f}{f(a)})^{\ast n}$. The sequential colimit of these types is contractible, because
the maps in this type sequence all factor through the unit type.
\end{proof}

Using the join construction, we can now give a new definition of the propositional truncation.

\begin{defn}\label{defn:proptrunc}
The \define{propositional truncation} $\trunc{-1}{A}$ of a type $A:\UU$ is defined to be 
sequential colimit of the type sequence
\begin{equation*}
\begin{tikzcd}
\emptyt \arrow[r] & A \arrow[r,"\inr"] & \join{A}{A} \arrow[r,"\inr"] & \join{A}{(\join{A}{A})} \arrow[r,"\inr"] & \cdots
\end{tikzcd}\qedhere
\end{equation*}
\end{defn}

\begin{cor}
The propositional truncation $\trunc{-1}{A}$ of a type $A$ is a mere proposition satisfying the universal property of propositional truncation. 
\end{cor}

In the following application of the join construction we characterize
the `canonical' idempotents of the join operation on maps. Note that in the
definition of canonical idempotent below, there is no special status for
$\inl:A\to \join[X]{A}{A}$ compared to $\inr:A\to\join[X]{A}{A}$. Indeed, the
maps $\inl$ and $\inr$ are homotopic, and therefore one of them is an
equivalence if and only if the other is an equivalence.

\begin{thm}\label{thm:idempotents}
For any map $f:A\to X$ in $\UU$ the following are equivalent:
\begin{enumerate}
\item $f$ is an embedding,
\item $f$ is a \define{canonical idempotent} for the join operation on maps, 
in the sense that the map $\inl:A\to\join[X]{A}{A}$ is an equivalence.
\end{enumerate}
\end{thm}

\begin{proof}
Recall that we have a commuting square
\begin{equation*}
\begin{tikzcd}[column sep=huge]
A \arrow[r,"\inl"] \arrow[d,swap,"\eqvsym"] & \join[X]{A}{A} \arrow[d,"\eqvsym"] \\
\sm{x:X}\hfib{f}{x} \arrow[r,swap,"\total{\inl}"] & \sm{x:X}\join{\hfib{f}{x}}{\hfib{f}{x}}.
\end{tikzcd}
\end{equation*}
It follows that $\inl:A\to\join[X]{A}{A}$ is an equivalence if and only if for
each $x:X$, the map $\inl:\hfib{f}{x}\to\join{\hfib{f}{x}}{\hfib{f}{x}}$ is an
equivalence. 
Since $f:A\to X$ is an embedding precisely when its fibers are mere 
propositions, we see that it suffices to prove the statement fiberwise.
More precisely, we show that for any $P:\UU$, the following are equivalent:
\begin{enumerate}
\item $P$ is a mere proposition,
\item $P$ is a \define{canonical idempotent} for the join operation on types, 
in the sense that the map $\inl:P\to \join{P}{P}$ is an equivalence.
\end{enumerate}
Suppose that $P$ is a mere proposition. 
Then $\join{P}{P}$ is a mere proposition, and we have $P\to\join{P}{P}$.
Moreover, we may use the universal property of the pushout to show that
$\join{P}{P}\to P$, since any two maps of type $P\times P\to P$ are equal. 
Therefore it follows that there is an equivalence $\eqv{P}{\join{P}{P}}$.
This shows that if $P$ is a mere proposition, then $P$ is an idempotent for
the join operation. Since any two maps of type $P\to\join{P}{P}$ are equal,
it also follows that $P$ is canonically idempotent.

For the converse, suppose that $\inl:P\to\join{P}{P}$ is an equivalence. 
Since we know that $P^{\ast\infty}$ is a mere proposition, 
we may show that $P$ is a mere proposition by constructing an equivalence of
type $\eqv{P}{P^{\ast\infty}}$. 
Since $P$ is the sequential colimit of the constant
type sequence at $P$, it suffices to show that the natural transformation
\begin{equation*}
\begin{tikzcd}
P \arrow[r,"{\idfunc[P]}"] \arrow[d,swap,"\inl"] & P \arrow[r,"{\idfunc[P]}"] \arrow[d,swap,"\inl"] & P \arrow[r,"{\idfunc[P]}"] \arrow[d,swap,"\inl"] & \cdots \\
P^{\ast 1} \arrow[r,swap,"\inr"] & P^{\ast 2} \arrow[r,swap,"\inr"] & P^{\ast 3} \arrow[r,swap,"\inr"] & \cdots
\end{tikzcd}
\end{equation*}
of type sequences is in fact a natural equivalence. In other words, we have
to show that for each $n:\N$, the map $\inl:P\to P^{\ast n}$ is an equivalence.

Of course, $\inl:P\to P^{\ast 1}$ is an equivalence. For the inductive step,
suppose that $\inl:P\to P^{\ast n}$ is an equivalence. First note that
we have a commuting triangle
\begin{equation*}
\begin{tikzcd}
P \arrow[rr,"\inl"] \arrow[dr,swap,"\inl"] & & \join{P}{P} \arrow[dl,"{\idfunc[P] \circledast\inl}"] \\
& P^{\ast(n+1)}
\end{tikzcd}
\end{equation*}
where $\idfunc[P] \circledast\inl$ denotes the functorial action of the join, applied
to $\idfunc[P]$ and $\inl:P\to P^{\ast n}$. Since both $\idfunc[P]$ and
$\inl:P\to P^{\ast n}$ are assumed to be equivalences, it follows that
$\idfunc[P]\circledast \inl$ is an equivalence. Therefore we get from the
$3$-for-$2$ rule that $\inl:P\to P^{\ast (n+1)}$ is an equivalence. 
\end{proof}

\section{The modified join construction}\label{sec:modified-join}

In this section we modify the join construction slightly, to construct the
image of a map $f:A\to X$ where we assume $X$ to be only locally small, rather than
small. To do this, we need to assume `global function extensionality', by which
we mean that function extensionality holds for all types, regardless of their {size
\footnote{Note that univalence implies function extensionality \emph{in} the universe $\UU$, but not necessarily global function extensionality.}}\footnote{In fact, we only need function extensionality for function types with a small domain and a locally small codomain.}.

We use the modified join construction to construct some classes of quotients
of low homotopy complexity: 
set-quotients and the Rezk completion of a precategory. 
We note that the modified join construction may also be used to construct
the $n$-truncation for any $n:\N$. 

\begin{defn}
A (possibly large) type $X$ is said to be \define{locally small} if 
for all $x,y:X$, there is a type $x='y:\UU$ and an equivalence
of type 
\begin{equation*}
\eqv{(x=y)}{(x='y)}.
\end{equation*} 
\end{defn}

\begin{rmk}
Being locally small in the above sense is a property, since in a univalent universe
any two witnesses of local smallness are equal.
\end{rmk}

\begin{eg}
Examples of locally small types include all types in $\UU$, 
the universe $\UU$ itself (by the univalence axiom), 
mere propositions of any size, 
and the exponent $A\to X$, for any $A:\UU$ and any locally small type $X$
(by global function extensionality).
\end{eg}

To construct the image of $f:A\to X$, mapping a small type $A$ into a locally
small type $X$, one can see that
the fibers of $f$ are equivalent to small types. Indeed, by the
local smallness condition, we have equivalences of type
\begin{equation*}
\eqv{\Big(\sm{a:A} f(a)=x\Big)}{\Big(\sm{a:A} f(a) =' x\Big)},
\end{equation*}
and the type on the right is small for every $x:X$. We will write
$\fib'{f}{x}$ for this modified fiber of $f$ at $x$. Since the modified fibers
are in $\UU$, we may $(-1)$-truncate them using \autoref{defn:proptrunc}. Therefore, we may define
\begin{equation}\label{eq:im_by_trunc}
\im'_t(f)\defeq\sm{x:X}\trunc{-1}{\fib'{f}{x}} 
\end{equation}
Of course, we have a commuting triangle
\begin{equation*}
\begin{tikzcd}
A \arrow[r] \arrow[dr,swap,"f"] & \im'_t(f) \arrow[d,"\proj 1"] \\
& X
\end{tikzcd}
\end{equation*}
with the universal property of the image inclusion of $f$, which follows from
Theorem 7.6.6 of \cite{hottbook}. 
Although this image exists under our working assumptions, 
it is not generally the case that $\im'_t(f)$ is a type in $\UU$.

One might also try the join construction directly to construct the image of
$f$. `The' join of two maps $f:A\to X$ and $g:B\to X$ into a locally small type
$X$ is formed taken by first taking the pullback of $f$ and $g$, and then the 
pushout of the two projections from the pullback.
However, the pullback of $f$ and $g$ is the type $\sm{a:A}{b:B}f(a)=g(b)$, and
this is not a type in $\UU$. Therefore, we may not just form the pushout
of $A \leftarrow (\sm{a:A}{b:B}f(a)=g(b)) \rightarrow B$. 
Hence we cannot follow the construction of the join of maps directly, in the
setting where we only assume $\UU$ to be closed under graph quotients.

Instead, we modify the definition of the join of maps, using the
condition that $X$ is locally small. By this condition we have
an equivalence of type $\eqv{(f(a)=g(b))}{(f(a)='g(b))}$, for any $a:A$ and 
$b:B$. It therefore follows that the type
\begin{equation}\label{eq:modified_pullback}
A\times'_X B \defeq \sm{a:A}{b:B}f(a)='g(b)
\end{equation} is
still the pullback of $f$ and $g$. We call this type the \define{modified
pullback} of $f$ and $g$.
Since each of the types $A$, $B$ and $f(a)='g(b)$ is in $\UU$, it follows that
the modified pullback $A\times_X'B$ is in $\UU$.

\begin{thm}\label{thm:modified-join}
Let $\UU$ be a univalent universe in Martin-L\"of type theory with global function extensionality, 
and assume that $\UU$ is closed under graph quotients. 

Let $A:\UU$ and let $X$ be any type which is locally small with respect to $\UU$.
Then we can construct a small type $\im'(f):\UU$, a surjective map $q'_f:A\to\im'(f)$, and an embedding $i'_f:\im'(f)\to X$ such that the triangle
\begin{equation*}
\begin{tikzcd}
A \arrow[r,"{q'_f}"] \arrow[dr,swap,"f"] & \im'(f) \arrow[d,"{i'_f}"] \\
& X
\end{tikzcd}
\end{equation*}
commutes, and $i_f:\im'(f)\to X$ has the universal property of the image inclusion of $f$, in the sense of \autoref{defn:universal}.
\end{thm}

\begin{proof}
We define the modified join $f \ast' g$ of $f$ and $g$ as the pushout of the
modified pullback, as indicated in the diagram
\begin{equation*}
\begin{tikzcd}
A\times_X' B \arrow[r,"\pi_2"] \arrow[d,swap,"\pi_1"] \arrow[dr, phantom, "{\ulcorner}", at end] & B \arrow[d,"\inr"] \arrow[ddr,bend left=15,"g"] \\
A \arrow[r,swap,"\inl"] \arrow[drr,bend right=15,swap,"f"] & A\ast_X' B \arrow[dr,densely dotted,swap,near start,"{f \ast' g}" xshift=1ex] \\
& & X.
\end{tikzcd}
\end{equation*}
Note that this is where we need to know that we can use the induction principle
of graph quotients to define maps from graph quotients into locally small types.

Now we can consider, for any $f:A\to X$ from $A:\UU$ into a locally small type
$X$, the modified join powers $f^{\ast'n}$ of $f$. The existence of each of
them follows from the assumption that $\UU$ is closed under graph quotients.
By an argument completely analogous to the argument given in the original join
construction, it follows that the sequential colimit
$i'_f\defeq f^{\ast'\infty}$ is an embedding with the universal property of the image 
inclusion of $f$. 
\end{proof}

\section{Direct applications of the modified join construction}

Recall that $\prop_\UU$ is the type $\sm{P:\UU}\isprop(P)$. A $\prop_\UU$-valued
equivalence relation on a type $A$, is a binary relation $R:A\to A\to\prop_\UU$
that is reflexive, symmetric and transitive in the expected sense.
A more thorough discussion on set-quotients can be found in \S 6.10 of
\cite{hottbook}.

\begin{cor}\label{cor:setquotients}
For any $\prop_\UU$-valued equivalence relation $R:A\to A\to\prop_\UU$ over a type
$A:\UU$, we get from the construction in \autoref{thm:modified-join} a type $A/R:\UU$
with the universal property of the quotient.
\end{cor}

\begin{proof}
In \S 10.1.3 of \cite{hottbook}, it is shown that the subtype 
\begin{equation*}
\sm{P:A\to\UU} \trunc{-1}{\sm{a:A} R(a)=P}
\end{equation*} 
of the type $A\to\prop_\UU$ has the universal property of the set-quotient.
Note that this is the image of $R$, as a function from $A$ to the locally small
type $A\to\prop_\UU$. 

Since the type $\im'(R):\UU$, which we obtain from \autoref{thm:modified-join},
has the universal property of the image, the universal property of the
set-quotient follows from an argument analogous to that given in \S 6.10 of \cite{hottbook}.
\end{proof}

By a small (pre)category $A$ we mean a (pre)category $A$ for which the type
$\mathsf{obj}(A)$ of objects is in $\UU$, and for which the type
$\mathsf{hom}_A(x,y)$ of morphisms from $x$ to $y$ is also in $\UU$, for any
$x,y:\mathsf{obj}(A)$. Pre-categories and Rezk-complete categories were introduced
in Homotopy Type Theory in \cite{AKS}.

\begin{cor}\label{cor:rezkcompletion}
The Rezk completion $\hat{A}$ of any small precategory $A$ can be constructed in any 
univalent universe that is closed under graph quotients,
and $\hat{A}$ is again a small category. 
\end{cor}

\begin{proof}
In the first proof of Theorem 9.9.5 of \cite{hottbook}, the Rezk completion of
a precategory $A$ is constructed as the image of the action on objects of the
Yoneda embedding $\mathbf{y}:A\to\mathbf{Set}^{\op{A}}$.

The hom-set $\mathbf{Set}^{\op{A}}(F,G)$ is the type of natural transformations
from $F$ to $G$. It is clear from Definition 9.9.2 of \cite{hottbook}, that the
type $\mathbf{Set}^{\op{A}}(F,G)$ is in $\UU$ for any two presheaves $F$ and $G$
on $A$. In particular, the type $F\cong G$ of isomorphisms from $F$ to $G$
is small for any two presheaves on $A$.

Since $\mathbf{Set}$ is a category, it follows from Theorem 9.2.5 of
\cite{hottbook} that the presheaf pre-category $\mathbf{Set}^{\op{A}}$ is a category.
Since the type of isomorphisms between any two objects is
small, it follows that the type of objects of
$\mathbf{Set}^{\op{A}}$ is locally small. 

Hence we can use \autoref{thm:modified-join} to construct the image of the
action on objects of the Yoneda-embedding. The image constructed in this way
is of course equivalent to the type $\hat{A}_0$ defined in the first proof of
Theorem 9.9.5. Hence the arguments presented in the rest of that proof apply
as well to our construction of the image. We therefore conclude that the Rezk completion
of any small precategory is a small category.
\end{proof}

\section{The join extension and connectivity theorems}

Some basic results about the join of maps include a generalization of Lemma 8.6.1 of
\cite{hottbook}, which we call the join extension theorem (\autoref{thm:join-extension}), and a closely
related theorem which we call the join connectivity theorem (\autoref{thm:join-connectivity}).
The idea of the join connectivity theorem came from Proposition 8.15 in 
Rezk's notes on homotopy toposes \cite{Rezk}.
We use the join connectivity theorem in 
\autoref{thm:joinconstruction-connectivity} to conclude that the connectivity 
of the approximations of the image inclusion increases.
In this sense, our approximating sequence of the image is very nice:
after $n$ steps of the approximation, only stuff of homotopy level higher than
$n$ is added.

Lemma 8.6.1 of \cite{hottbook} states that if $f:A\to B$ is an $m$-connected map,
and if $P:B\to\UU$ is a family of $(m+n+2)$-truncated types,
then precomposing by $f$ gives an $n$-truncated map of type
\begin{equation*}
\Big(\prd{b:B}P(b)\Big)\to\Big(\prd{a:A}P(f(a))\Big).
\end{equation*}
The general join extension theorem states that if $f:A\to B$ is an $M$-connected
map for some type $M$, and $P:B\to\UU$ is a family of $(\join{M}{N})$-local 
types, then the mentioned precomposition is an $N$-local map 
(we recall the terminology shortly). 
Note that, if one takes spheres $\Sn^m$ and $\Sn^n$ for $M$ and $N$, 
one retrieves Lemma 8.6.1 of \cite{hottbook} as a simple corollary.

We conclude this section with \autoref{thm:joinconstruction-connectivity},
asserting that $f^{\ast n}$ factors through an $(n-2)$-connected map to
$\im(f)$, for each $n:\N$.

\begin{defn}\label{defn:local}
or a given type $M$, a type $A$ is said to be \define{$M$-local} if the map
\begin{equation*}
\lam{a}{x}a : A \to (M \to A)  
\end{equation*}
is an equivalence.
\end{defn}

In other words, the type $A$ is $M$-local if each $f:M\to A$ has a unique extension along the
map $M\to\unit$, as indicated in the diagram
\begin{equation*}
\begin{tikzcd}
M \arrow[r,"f"] \arrow[d] & A \\
\unit. \arrow[ur,densely dotted]
\end{tikzcd}
\end{equation*}
Note that being $M$-local in the above sense is a mere proposition, so that the
type of all $M$-local types in $\UU$ is a subuniverse of $\UU$%
\footnote{When $\UU$ is assumed to be closed under recursive higher inductive
types, there exists an operation $\modal_M : \UU\to\UU$, 
which maps a type $A$ to the universal $M$-local type $\modal_M(A)$
with a map of type $A\to\modal_M(A)$. This operation is called 
\define{localization at $M$}, and it is a modality. 
This is just a special case of localization. There is a more general
notion of localization at a family of maps, see%
~\cite{RijkeShulmanSpitters}, for which the localization operation
is a reflective subuniverse, but \emph{not generally} a modality.%
}.

Dually, a type $X$ is said to be \define{$M$-connected} if for every $M$-local
type $A$, the map
\begin{equation*}
\lam{a}{x}{a} : A \to (X \to A)
\end{equation*}
is an equivalence. Thus in particular, $M$ itself is $M$-connected. 
Equivalently, a type $X$ is $M$-connected if the type
$\modal_M(X)$ is contractible. The survey article \cite{RijkeShulmanSpitters} contains much
more information about local types and the operation of localization.
In the present article, we focus on the interaction
of the join operation with the notions of being local and of connectedness.

\begin{defn}
Let $M$ be a type. We say that a type $X$ has the \define{$M$-extension property}
with respect to a map $F:A\to B$, if the map
\begin{equation*}
\lam{g}{a} g(F(a)) : (B\to X)\to (A\to X)
\end{equation*}
is $M$-local. In the case $M\jdeq\unit$, we say that $X$ is $F$-local.
\end{defn}

Usually, a type $A$ is said to be $M$-connected if its localization
$\modal_M(A)$ is contractible. 
Since we have not assumed that the universe is closed under a general class of recursive higher inductive types, we cannot simply
assume that the operation
$\modal_M:\UU\to\UU$ is available. 
Therefore we give a definition of connectedness which is equivalent to the
usual one in the presence of this operation.

\begin{defn}\label{defn:connected}
We say that a type $A$ is \define{$M$-connected} if any $M$-local type is $A$-local. 
A map $f:A\to X$ is said to be $M$-connected if its fibers are $M$-connected.
\end{defn}

\begin{lem}\label{lem:equivalent-extension-problems}
For any three types $A$, $A'$ and $B$, the type $B$ is $(\join{A}{A'})$-local
if and only if for any any $f:A\to B$, the type
\begin{equation*}
\sm{b:B}\prd{a:A}f(a)=b
\end{equation*}
is $A'$-local.
\end{lem}

\begin{proof}
To give $f:A\to B$ and $(f',H):A'\to\sm{b:B}\prd{a:A}f(a)=b$ is equivalent to giving a map $g:\join{A}{A'}\to B$. Concretely, the equivalence is given by substituting in $g:\join{A}{A'}\to B$ the constructors of the join, to obtain $\pairr{g\circ\inl,g\circ\inr,\apfunc{g}\circ\glue}$. 

Now observe that the fiber of precomposing with the unique map $!_{\join{A}{A'}} : \join{A}{A'}\to\unit$ at $g : \join{A}{A'}\to B$, is equivalent to
\begin{equation*}
\sm{b:B}\prd{t:\join{A}{A'}}g(t)=b.
\end{equation*}
Similarly, the fiber of precomposing with the unique map $!_{A'} : A'\to\unit$ at $\pairr{g\circ\inr,\apfunc{g}\circ\glue} : A'\to\sm{b:B}\prd{a:A}f(a)=b$ is equivalent to
\begin{equation*}
\sm{b:B}{h:\prd{a:A}g(\inl(a))=b}\prd{a':A'}\pairr{g(\inr(a')),\apfunc{g}(\glue(a,a'))}=\pairr{b,h}.
\end{equation*}
By the universal property of the join, these types are equivalent.
\end{proof}

\begin{lem}\label{lem:join-local}
Suppose $A$ is an $M$-connected type, and that $B$ is an $(\join{M}{N})$-local type. Then $B$ is $(\join{A}{N})$-local.
\end{lem}

\begin{proof}
Let $B$ be a $(\join{M}{N})$-local type. Our goal of showing that $B$ is
$(\join{A}{N})$-local is equivalent to showing that for any $f:N\to B$, 
the type 
\begin{equation*}
\sm{b:B}\prd{a:A}f(a)=b
\end{equation*}
is $A$-local. 
Since $B$ is assumed to be $(\join{M}{N})$-local, we know that this type is 
$M$-local. Since $A$ is $M$-connected, this type is also $A$-local.
\end{proof}

\begin{lem}\label{lem:N-extension-simple}
Let $A$ be $M$-connected and let $B$ be $(\join{M}{N})$-local. Then the map
\begin{equation*}
\lam{b}{a}b:B\to B^A
\end{equation*}
is $N$-local. 
\end{lem}

\begin{proof}
The fiber of $\lam{b}{a}b$ at a function $f:A\to B$ is equivalent to the type $\sm{b:B}\prd{a:A}f(a)=b$. Therefore, it suffices to show that this type is $N$-local. By \autoref{lem:equivalent-extension-problems}, it is equivalent to show that $B$ is $(\join{A}{N})$-local. This is solved in \autoref{lem:join-local}.
\end{proof}

\begin{thm}[Join extension theorem]\label{thm:join-extension}
Suppose $f:X\to Y$ is $M$-connected, and let $P:Y\to\UU$ be a family of
$(\join{M}{N})$-local types for some type $N$. Then precomposition by $f$, i.e.
\begin{equation*}
\lam{s}s\circ f : \Big(\prd{y:Y}P(y)\Big)\to\Big(\prd{x:X}P(f(x))\Big),
\end{equation*}
is an $N$-local map.
\end{thm}

\begin{proof}
Let $g:\prd{x:X}P(f(x))$. Then we have the equivalences
\begin{align*}
\hfib{(\blank\circ f)}{g} 
& \eqvsym \sm{s:\prd{y:Y}P(y)}\prd{x:X}s(f(x))=g(x) \\
& \eqvsym \sm{s:\prd{y:Y}P(y)}\prd{y:Y}{(x,p):\hfib{f}{y}} s(y)= \trans{p}{g(x)} \\
& \eqvsym \prd{y:Y}\sm{z:P(y)}\prd{(x,p):\hfib{f}{y}} \trans{p}{g(x)}=z \\
& \eqvsym \prd{y:Y}\hfib{\lam{z}{(x,p)}z}{\lam{(x,p)}\trans{p}{g(x)}}.
\end{align*}
Therefore, it suffices to show for every $y:Y$, that $P(y)$ has the $N$-extension property with respect to the unique map of type $\hfib{f}{y}\to\unit$. This is a special case of \autoref{lem:N-extension-simple}.
\end{proof}

\begin{thm}\label{thm:simple-join}
Suppose $X$ is an $M$-connected type and $Y$ is an $N$-connected type. Then $\join{X}{Y}$ is an $(\join{M}{N})$-connected type.
\end{thm}

\begin{proof}
It suffices to show that any $(\join{M}{N})$-local type is $(\join{X}{Y})$-local.
Let $Z$ be an $(\join{M}{N})$-local type.
Since $Z$ is assumed to be $(\join{M}{N})$-local, it follows by \autoref{lem:join-local} that $Z$ is $(\join{X}{N})$-local. By symmetry of the join, it also follows that $Z$ is $(\join{X}{Y})$-local.
\end{proof}

\begin{thm}[Join connectivity theorem]\label{thm:join-connectivity}
Consider an $M$-connected map $f:A\to X$ and an $N$-connected map $g:B\to X$. Then $\join{f}{g}$ is $(\join{M}{N})$-connected.
\end{thm}

\begin{proof}
This follows from \autoref{thm:simple-join} and \autoref{defn:join-fiber}.
\end{proof}

\begin{thm}\label{thm:joinconstruction-connectivity}
Consider the factorization
\begin{equation*}
\begin{tikzcd}
A_n \arrow[dr,swap,"f^{\ast n}"] \arrow[r,"q_n"] & \im(f) \arrow[d] \\
& X
\end{tikzcd}
\end{equation*}
of $f^{\ast n}$ through the image $\im(f)$. 
Then the map $q_n$ is $(n-2)$-connected, for each $n:\N$.
\end{thm}

\begin{proof}
We first show the assertion that, given a commuting diagram of the form
\begin{equation*}
\begin{tikzcd}
A \arrow[r,"q"] \arrow[dr,swap,"f"] & Y \arrow[d,"m"] & A' \arrow[l,swap,"{q'}"] \arrow[dl,"{f'}"] \\
& X
\end{tikzcd}
\end{equation*}
in which $m$ is an embedding, then $\join{f}{f'}=\join{(m\circ q)}{(m\circ q')}=m\circ (\join{q}{q'})$.
In other words, postcomposition with embeddings distributes over 
the join operation.

Note that, since $m$ is assumed to be an embedding, we have an equivalence of
type $\eqv{f(a)=f'(a)}{q(a)=q'(a)}$, for every $a:A$. Hence the pullback of
$f$ and $f'$ is equivalent to the pullback of $q$ along $q'$. Consequently, the
two pushouts
\begin{equation*}
\begin{tikzcd}
A\times_X A' \arrow[r,"\pi_2"] \arrow[d,swap,"\pi_1"] & A' \arrow[d] \\
A \arrow[r] & \join[X]{A}{A'}
\end{tikzcd}
\qquad\text{and}\qquad
\begin{tikzcd}
A\times_Y A' \arrow[r,"\pi_2"] \arrow[d,swap,"\pi_1"] & A' \arrow[d] \\
A \arrow[r] & \join[Y]{A}{A'}
\end{tikzcd}
\end{equation*}
are equivalent. Hence the claim follows.

As a corollary, we get that $q_n=q_f^{\ast n}$. Note that $q_f$ is surjective,
in the sense that $q_f$ is $\bool$-connected, where $\bool$ is the type of booleans%
\footnote{Recall that the $\bool$-local types are precisely the mere propositions.}.
Hence it follows that $q_n$ is $\bool^{\ast n}$-connected. 

Now recall that the $n$-th join power of $\bool$ is the $(n-1)$-sphere $\Sn^{n-1}$, and that
a type is $(\Sn^{n-1})$-connected if and only if it is $(n-2)$-connected.
\end{proof}

\section{The construction of the $n$-truncation}\label{sec:truncation}

In this section we will construct for any $n:\N$, the $n$-truncation on any univalent universe that contains
a natural numbers object and is closed under graph quotients.
We will do this via the modified join construction of \autoref{thm:modified-join}.
Recall that a $(-2)$-truncated type is simply a contractible type, and that
for $n\geq -2$ an $(n+1)$-truncated type is a type of which the identity types
are $n$-truncated. The $(-2)$-truncation is easy to construct: it sends
every type to the unit type $\unit$. Thus, we shall proceed by induction
on the integers greater or equal to $-2$, and assume that the universe admits
an $n$-truncation operation $\trunc{n}{\blank}:\UU\to\UU$ for a given $n$.

A suggestive way to think of the type $\trunc{n+1}{A}$ is as the quotient of $A$ modulo the
`$(n+1)$-equivalence relation' given by $\trunc{n}{a=b}$. 
Indeed, by Theorem 7.3.12 of \cite{hottbook} we have that the canonical map
\begin{equation*}
\trunc{n}{a=b}\to(\tproj{n+1}{a}=\tproj{n+1}{b})
\end{equation*}
is an equivalence, and the unit $\tproj{n+1}{\blank}:A\to \trunc{n+1}{A}$ is
a surjective map (it is in fact $(n+1)$-connected). 

\begin{thm}\label{thm:truncation}
In Martin-L\"of type theory with a univalent universe $\UU$ that is closed under
graph quotients we can define, for every $n\geq -2$, an $n$-truncation operation
\begin{equation*}
\trunc{n}{\blank} : \UU\to\UU
\end{equation*}
and for every $A:\UU$ a map
\begin{equation*}
\tproj{n}{\blank}:A\to\trunc{n}{A},
\end{equation*}
such that for each $A:\UU$ the type $\trunc{n}{A}$ is an $n$-truncated type satisfying the (dependent) universal property of $n$-truncation, that for every $P:\trunc{n}{A}\to\UU$ such that every $P(x)$ is $n$-truncated,
the canonical map
\begin{equation*}
\blank\circ\tproj{n}{\blank} : \Big(\prd{x:\trunc{n}{A}}P(x)\Big)\to\Big(\prd{a:A}P(\tproj{n}{a})\Big)
\end{equation*}
is an equivalence.
\end{thm}

\begin{proof}[Construction]
As announced, we define the $n$-truncation operation by induction on $n\geq-2$,
with the trivial operation as the base case. Let $n:\N$ and suppose we have
an $n$-truncation operation as described in the statement of the theorem.

We first define the reflexive relation $\mathscr{Y}_n(A) : A \to A \to \UU$ by
\begin{equation*}
\mathscr{Y}_n(A)(a,b) \defeq \trunc{n}{a=b}.
\end{equation*}
Note that the codomain $(A\to\UU)$ of $\mathscr{Y}_n(A)$ is locally small since it is the exponent of
the locally small type $\UU$ by a small type $A$. Hence we we obtain the image
of $\mathscr{Y}_n(A)$ from the modified join construction of \autoref{thm:modified-join}.
This allows us to define
\begin{align*}
\trunc{n+1}{A} & \defeq \im'(\mathscr{Y}_n(A)) \\
\tproj{n+1}{\blank} & \defeq q'_{\mathscr{Y}_n(A)}
\end{align*}
For notational reasons, we shall just write $\im(\mathscr{Y}_n(A))$ for $\im'(\mathscr{Y}_n(A))$. 

We will show that $\trunc{n+1}{A}$ is indeed $(n+1)$-truncated in \autoref{cor:truncated} of \autoref{lem:modal_contr} below. Once this fact is established, it remains to verify the dependent universal property of $(n+1)$-truncation.
By the join extension theorem \autoref{thm:join-extension} (using $N\defeq \emptyt$), it suffices to show that the map $\tproj{n+1}{\blank}:A\to\trunc{n+1}{A}$ is $\sphere{n+2}$-connected. Note that $\tproj{n+1}{\blank}$ is surjective, so the claim that $\tproj{n+1}{\blank}$ is $\sphere{n+2}$-connected follows from \autoref{lem:ap_connectivity}, where we show that for any surjective map $f:A\to X$, if the action on paths is $M$-connected for any two points in $A$, then $f$ is $\susp(M)$-connected. To apply this lemma, we also need to know that $\tproj{n}{\blank}:A\to\trunc{n}{A}$ is $\sphere{n+1}$-connected. This is shown in Corollary 7.5.8 of \cite{hottbook}.
\end{proof}

Before we prove that $\im(\mathscr{Y}_n(A))$ is $(n+1)$-truncated, we prove the stronger claim that $\im(\mathscr{Y}_n(A))$ has the desired identity types:

\begin{lem}\label{lem:modal_contr}
For every $a,b:A$, we have an equivalence
\begin{equation*}
\eqv{\trunc{n}{a=b}}{(\mathscr{Y}_n(A)(a)=\mathscr{Y}_n(A)(b))}.
\end{equation*}
\end{lem}

\begin{proof}
To characterize the identity type of $\im(\mathscr{Y}_n(A))$ we wish to apply
the encode-decode method. Thus, we need to provide for every $b:A$ a type 
family $Q_b:\im(\mathscr{Y}_n(A))\to\UU$ with a point $q_b:Q_b(\mathscr{Y}_n(A)(b))$,
such that the total space
\begin{equation*}
\sm{P:\im(\mathscr{Y}_n(A))} Q_b(P)
\end{equation*}
is contractible. Moreover, it must be the case that $\eqv{Q_b(\mathscr{Y}_n(A)(a))}{\trunc{n}{a=b}}$ for any $a:A$. 

To construct $Q_b$, note that for any $b:A$, the image inclusion $i:\im(\mathscr{Y}_n(A))\to (A\to\UU)$ defines 
a type family $Q_b:\im(\mathscr{Y}_n(A))\to\UU$ by $Q_b(P)\defeq P(b)$. With this definition for $Q_b$ it follows that $Q_b(\mathscr{Y}_n(A)(a))\jdeq\mathscr{Y}_n(A)(a,b)\jdeq\trunc{n}{a=b}$, as desired. Moreover, we have a reflexivity term $\tproj{n}{\refl{b}}$ in $\trunc{n}{b=b}$, so it remains to prove that the total space 
\begin{equation*}
\sm{P:\im(\mathscr{Y}_n(A))}P(b)
\end{equation*}
of $Q_b$ is contractible. For the center of contraction we take the pair
$\pairr{\mathscr{Y}_n(A)(b),\tproj{n}{\refl{b}}}$.
Now we need to construct a term of type
\begin{equation*}
\prd{P:\im(\mathscr{Y}_n(A))}{y:P(b)} \pairr{\mathscr{Y}_n(A)(b),\tproj{n}{\refl{b}}}=\pairr{P,y}.
\end{equation*}
Since $\mathscr{Y}_n(A)(b,a)\jdeq\trunc{n}{b=a}$, it is equivalent to construct a term of type
\begin{equation*}
\prd{P:\im(\mathscr{Y}_n(A))}{y:P(b)}\sm{\alpha:\prd{a:A} \eqv{\trunc{n}{b=a}}{P(a)}} \alpha_b(\tproj{n}{\refl{b}})=y.
\end{equation*}
Let $P:\im(\mathscr{Y}_n(A))$ and $y:P(b)$. Then $P(a)$ is $n$-truncated for any $a:A$. Therefore, to construct a map
$\alpha(P,y)_a:\trunc{n}{b=a}\to P(a)$, it suffices to construct a map of type $(b=a)\to P(a)$. This may be done by
path induction, using $y:P(b)$. Since it follows that $\alpha(P,y)_b(\tproj{n}{\refl{b}})=y$, it only remains to show that each $\alpha(P,y)_a$ is an equivalence.  

Note that the type of those $P:\im(\mathscr{Y}_n(A))$ such that for all $y:P(b)$ and all $a:A$ the map $\alpha(P,y)_a$ is an equivalence, is a subtype of $\im(\mathscr{Y}_n(A))$, we may use the universal property of the image of $\mathscr{Y}_n(A)$: it suffices to lift
\begin{equation*}
\begin{tikzcd}
& \sm{P:\im(\mathscr{Y}_n(A))}\prd{y:P(b)}{a:A}\isequiv(\alpha(P,y)_a) \arrow[d] \\
A \arrow[ur,densely dotted] \arrow[r,swap,"\mathscr{Y}_n(A)"] & \im(\mathscr{Y}_n(A)).
\end{tikzcd}
\end{equation*}
In other words, it suffices to show that 
\begin{equation*}
\prd{x:A}{y:\mathscr{Y}_n(A)(x,b)}{a:A}\isequiv(\alpha(\mathscr{Y}_n(A)(x),y)_a).
\end{equation*}
Thus, we want to show that for any $y:\trunc{n}{x=b}$, the map $\trunc{n}{a=b}\to\trunc{n}{x=b}$ constructed above is an equivalence.
Since the fibers of this map are $n$-truncated, and $\iscontr(X)$ of an $n$-truncated type $X$ is always $n$-truncated, we may assume that $y$ is of the form $\tproj{n}{p}$ for $p:x=b$. 
Now it is easy to see that our map of type $\trunc{n}{b=a}\to\trunc{n}{x=a}$ is the unique map which
extends the path concatenation $\ct{p}{\blank}$, as indicated in the diagram
\begin{equation*}
\begin{tikzcd}[column sep=8em]
(b=a) \arrow[r,"\ct{p}{\blank}"] \arrow[d] & (x=a) \arrow[d] \\
\trunc{n}{b=a} \arrow[r,densely dotted,swap,"{\alpha(\mathscr{Y}_n(A)(x),y)_a}"] & \trunc{n}{x=a}.
\end{tikzcd}
\end{equation*}
Since the top map is an equivalence, it follows that the map $\alpha(\mathscr{Y}_n(A)(x),y)_a$ is an equivalence.
\end{proof}

\begin{cor}\label{cor:truncated}
The image $\im(\mathscr{Y}_n(A))$ is an $(n+1)$-truncated type. 
\end{cor}

Before we are able to show that for any surjective map $f:A\to X$, if the action on paths is $M$-connected for any two points in $A$, then $f$ is $\susp(M)$-connected, we show that a type is $\susp(M)$-connected precisely when its identity types are $M$-connected.

\begin{lem}\label{lem:local_id}
Let $M$ be a type. Then a type $X$ is $(\join{\bool}{M})$-local
if and only if all of its identity types are $M$-local. 
\end{lem}

\begin{proof}
The map
\begin{equation*}
\lam{p}{m}p : (x=y)\to (M\to (x=y))
\end{equation*}
is an equivalence if and only if the induced map on total spaces
\begin{equation*}
\lam{\pairr{x,y,p}}\pairr{x,y,\lam{m}p} : \Big(\sm{x,y:X}x=y\Big)\to\Big(\sm{x,y:X}M\to (x=y)\Big)
\end{equation*}
is an equivalence. 
Since the map $\lam{x}\pairr{x,x,\refl{x}}:X\to\sm{x,y:X}x=y$ is an equivalence,
the above map is an equivalence if and only if the map
\begin{equation*}
\lam{x}\pairr{x,x,\lam{m}\refl{x}} : X\to\Big(\sm{x,y:X}M\to (x=y)\Big)
\end{equation*}
is an equivalence. For every $x:X$, the triple $\pairr{x,x,\lam{m}\refl{x}}$
induces a map $\susp(M)\to X$. By uniqueness of the universal property,
it follows that this map is the constant map $\lam{m}x$.
Thus we see that $\lam{x}\pairr{x,x,\lam{m}\refl{x}}$ is an equivalence if
and only if the map
\begin{equation*}
\lam{x}{m}x : X \to (\susp(M)\to X)
\end{equation*}
is an equivalence. 
\end{proof}

\begin{lem}\label{lem:ap_connectivity}
Suppose $f:A\to X$ is a surjective map, with the property that for every
$a,b:A$, the map
\begin{equation*}
\mapfunc{f}(a,b):(a=b)\to (f(a)=f(b))
\end{equation*}
is $M$-connected. Then $f$ is $\susp(M)$-connected. 
\end{lem}

\begin{proof}
We have to show that $\fib{f}{x}$ is $\susp(M)$-connected for each $x:X$. 
Since this is a mere proposition, and we assume that $f$ is surjective, it
is equivalent to show that $\fib{f}{f(a)}$ is $\susp(M)$-connected for each $a:A$. 
Let $Y$ be a $\susp(M)$-local type. 
For every $g:\fib{f}{f(a)}\to Y$ be a map we have the point $\theta(g)\defeq g(a,\refl{f(a)})$ in $Y$,
so we obtain a map
\begin{equation*}
\theta : (\fib{f}{f(a)}\to Y)\to Y
\end{equation*}
It is clear that $\theta(\lam{\pairr{b,p}}y)=y$, so it remains to show that
for every $g:\fib{f}{f(a)}\to Y$ we have $\lam{\pairr{b,p}}\theta(g)=g$.
That is, we must show that
\begin{equation*}
\prd{b:A}{p:f(a)=f(b)} g(a,\refl{f(a)})=g(b,p).
\end{equation*}
Using the assumption that $Y$ is $\susp(M)$-connected, it follows from
\autoref{lem:local_id} that the type $g(a,\refl{f(a)})=g(b,p)$ is $M$-connected,
for every $b:A$ and $p:f(a)=f(b)$.
Therefore it follows, since the map $\mapfunc{f}(a,b):(a=b)\to(f(a)=f(b))$ is connected, that our goal is equivalent to
\begin{equation*}
\prd{b:A}{p:a=b} g(a,\refl{f(a)})=g(b,\mapfunc{f}(a,b,p)).
\end{equation*}
This follows by path induction. 
\end{proof}

\section{Summary and conclusion}

In this paper we have worked in Martin-L\"of's dependent type theory with
global function extensionality, with a univalent universe which is closed
under graph quotients satisfying a global induction principle. 

In our main theorem (\autoref{thm:modified-join}) 
we showed that for any $f:A\to X$ with $A:\UU$ and $X$ locally small with respect to $\UU$,
the image of $f$ can be constructed in $\UU$. 
We used this construction of images to construct 
set-quotients (\autoref{cor:setquotients}), 
and the Rezk completion of a precategory (\autoref{cor:rezkcompletion}),
and we showed that the $n$-truncations can be described in any universe which is closed under graph quotients (\autoref{thm:truncation}). 
We note that in these constructions we need neither a propositional resizing axiom as proposed in \cite{hottbook}, nor recursive higher inductive types, nor higher inductive types with
higher path constructors.

Via a generalization of the construction of the $n$-truncation, we are also able to construct for any
modality $\modal$, the modality of $\modal$-separated types. For the definition
of modality in homotopy type theory we refer the reader to \S 7.7 of \cite{hottbook},
and the details of the construction of the modality of $\modal$-separated types
will appear in \cite{RijkeShulmanSpitters}. 

\mentalpause
\paragraph*{\bf Acknowledgments.}
I would like to thank Steve Awodey and Ulrik Buchholtz for the many insightful discussions, and for carefully reading the drafts of this paper.
I gratefully acknowledge the support of the Air Force Office of Scientific Research through MURI grant FA9550-15-1-0053.

\bibliographystyle{plain}
\bibliography{refs}

\end{document}